\documentclass[letterpaper, 10 pt, conference]{ieeeconf}  
                                                          
\usepackage{amsmath, amssymb, amsfonts}
\usepackage[english]{babel}
\usepackage[T1]{fontenc}
\usepackage{epsfig}
\usepackage{array}
\usepackage{subfigure}
\usepackage{color}

\newtheorem{definition}{Definition}
\newtheorem{remark}{Remark}

\providecommand{\abs[1]}{\left|#1\right|}
\newcommand\real{\ensuremath{{\mathbb R}}}

\newenvironment{proofof}{\noindent {\em Proof of }}{\hfill \hspace*{1pt}
\hfill $\blacksquare$}

\newtheorem{lem}{Lemma} 
\newenvironment{lemma}{\begin{lem}\rm }{\hfill \hspace*{1pt} \hfill $\diamond$\end{lem}}

\newtheorem{prop}{Proposition}
\newenvironment{proposition}{\begin{prop} \rm }{\hfill \hspace*{1pt} \hfill $\diamond$\end{prop}}

\newcommand{\calR}{\mathcal{R}}

\newcommand{\calK}{\mathcal{K}}

\newcommand{\calX}{\mathcal{X}}

\makeindex

\IEEEoverridecommandlockouts                              
\title{\LARGE \bf An operator-theoretic approach to differential positivity
}
\author{A. Mauroy,
\thanks{A. Mauroy is with the Department of Electrical Engineering and Computer Science,
        University of Liège, 4000 Liège, Belgium,
        {\tt\small a.mauroy@ulg.ac.be}}
				F. Forni, and R. Sepulchre%
\thanks{F. Forni and R. Sepulchre are with the University of Cambridge, Department of Engineering, Trumpington Street, Cambridge CB2 1PZ, and with the Department of Electrical Engineering and Computer Science, University of Liège, 4000 Liège, Belgium, {\tt\small ff286@cam.ac.uk, r.sepulchre@eng.cam.ac.uk}}%
}

\begin{document}

\maketitle
\thispagestyle{empty}
\pagestyle{empty}

\begin{abstract}

Differentially positive systems are systems 
whose linearization along trajectories is positive. 
Under mild assumptions, their solutions asymptotically converge to a one-dimensional attractor, which must be a limit cycle in the absence of fixed points in the limit set. In this paper, we investigate the general connections between the (geometric) properties of differentially positive systems and the (spectral) properties of the Koopman operator. In particular, we obtain converse results for differential positivity, showing for instance that any hyperbolic limit cycle is differentially positive in its basin of attraction. We also provide the construction of a contracting cone field.

\end{abstract}

\section{Introduction}
{\color{black}

A linear system is positive 
if its trajectories leave
some conic subset $\calK$ of the state space 
invariant \cite{Bushell}. Positivity is
at the core of a number of applications
because it strongly restricts 
the linear behavior \cite{Farina2000,Moreau2004,Roszak2009}. 
Under mild conditions,
Perron-Frobenius theory guarantees that 
every bounded trajectory converges asymptotically
to a one-dimensional attractor given by the ray
$\lambda \mathbf{w}\!\!\in\!\! \calK$, 
where $\lambda \!\!\in\!\! \real$ and $\mathbf{w}$ is the
dominant eigenvector of the system state matrix
\cite{Birkhoff1957,Bushell}.

Differential positivity 
brings linear positivity to the nonlinear setting.
A nonlinear system is differentially positive if 
its \emph{linearization} 
along any trajectory 
leaves some cone (field) invariant \cite{Forni_diff_pos}.
Differentially positive systems are a sizable class of
systems, which includes monotone systems 
\cite{Angeli2003,Hirsch1995,Smith1995}.
Differential positivity also
restricts the asymptotic behavior of a nonlinear system.
Under mild conditions, a suitable differential 
formulation of Perron-Frobenius theory 
guarantees that the trajectories of the nonlinear
systems converge asymptotically to a one-dimensional attractor. In contrast to linear positivity,
this attractor is not necessarily a ray but a curve,
possibly given by a collection of fixed points
and connecting arcs or by a limit cycle \cite{Forni_diff_pos,Forni2015}.

Differential positivity is a promising tool
for the study of bistable and periodic behaviors,
since it reduces the analysis of those behaviors to
the characterization of a suitable cone field 
on the system state manifold. However,
besides specific families of monotone systems,
there is no constructive methodology to find
such cone fields. At more fundamental level,
it is not even clear how demanding is to 
use differential positivity for capturing bistable and
periodic behaviors. This paper provides a first
answer to both these questions. 

Bridging 
the (geometric) properties of differentially positive systems and the (spectral) properties of the Koopman operator \cite{Mezic}, 
the paper illustrates the tight connection between the
existence of a suitable collection of Koopman eigenfunctions
for the system and the construction of a cone field.
This approach leads directly to a 
numerical tool for constructing cone fields.
At more fundamental level, Koopman theory provides a way to 
derive converse results for differential positivity.
The striking outcome is that any system with 
a hyperbolic limit cycle is differentially positive
in the basin of attraction of the limit cycle.

A short introduction to both differential positivity and
Koopman operator theory is provided in the next section,
which follows a brief discussion on the basic geometric tools 
used in the paper.
The connection between Koopman operator theory 
and differential positivity is developed
in Section \ref{sec_gen_results}}. We
provide an explicit construction of the cone field
based on a suitable set of Koopman eigenfunctions, 
and we show the precise relation between 
the so-called Perron-Frobenius vector field and 
the dominant Koopman eigenfunction.
Section \ref{sec_converse} is dedicated to converse results
for hyperbolic fixed points and hyperbolic limit cycles.
Section \ref{sec_numerical} 
provides a cone field
for a system with a stable equilibrium and for the Van der Pol oscillator,
by exploiting numerical methods for 
computing Koopman eigenfunctions
based on Laplace averages. 
Conclusions follow. Proofs are in appendix.

{\color{black} 
\section{A glimpse into differential positivity and Koopman operator}

\subsection{Manifolds and prolonged dynamics}

The exposition of the paper takes advantage of a few basic geometric notions on Riemannian manifolds.
Let $\mathcal{X}$ be a smooth $n$-dimensional manifold endowed with 
a Riemannian metric $\langle \cdot,\cdot \rangle_x:T_x\mathcal{X} \times T_x\mathcal{X} \to \mathbb{R}$
where $T_x\mathcal{X}$ denotes the tangent space at $x \in \mathcal{X}$. 
We will use $|\delta x|$ to denote $\sqrt{\langle \delta x, \delta x \rangle_x}$ for all $\delta x\in T_x \mathcal{X}$. 
Given two smooth manifolds $\mathcal{X}_1$ and $\mathcal{X}_2$, and a differentiable function (or observable)
$g: \mathcal{X}_1 \to \mathcal{X}_2$, 
let $\partial g(x): T_x\mathcal{X}_1 \to T_{g(x)}\mathcal{X}_2$,
$\delta x \mapsto \partial g(x)[\delta x]$, be the differential of $g$ at $x$.
When clear from the context, we will simply write $\partial g(x)\delta x \triangleq \partial g(x) [\delta x]$,
and we will use $\partial g(x) A := \{\partial g(x) \delta x | \delta x \in A\}$ for all $A \subseteq T_x\mathcal{X}$.

The paper focuses on continuous-time dynamical systems $\Sigma$ on $\mathcal{X}$ 
represented by $\dot{x}=f(x)$, where 
$x \in \mathcal{X}$ and $f(x) \in T_x\mathcal{X}$. 
We assume that $f\in C^2 (\mathcal{X})$ (twice differentiable)
and that the system is forward and backward complete, that is, 
the flow $\psi:\mathbb{R} \times \mathcal{X} \to \mathcal{X}$ of $\Sigma$
satisfies $\frac{d}{dt} \psi(t,x) = f(\psi(t,x))$ for all $t \in \mathbb{R}$ and $x\in \mathcal{X}$. 
In what follows, for simplicity, we will also use the mapping $\psi^t(\cdot) \triangleq \psi(t,\cdot): \mathcal{X} \to \mathcal{X}$
and we will sometimes refer to the trajectory $x(\cdot)\triangleq\psi(\cdot,x_0)$ of $\Sigma$ from the initial condition $x(t_0)=\psi(t_0,x_0)$.


To characterize the property of differential positivity
we will make use of the notion of prolonged dynamics
$\delta \Sigma$ of $\Sigma$, represented by
\begin{equation*}
\label{eq:prolonged_system}
\delta \Sigma:\left \{ \begin{array}{rcl} \dot{x} & = & f(x) \\
\dot{\delta x} & = & \partial f (x) \delta x
\end{array} \right.
\end{equation*}
where $(x,\delta x)$ belongs to the tangent bundle $T \mathcal{X} = \bigcup_{x \in \mathcal{X}} \{x\} \times T_x \mathcal{X}$, \cite{Crouch1987}. 
The flow of $\delta \Sigma$ is a mapping in $\mathbb{R} \times T\mathcal{X} \to T\mathcal{X}$ 
specified by $(x,\delta x) \mapsto (\psi^t(x),\partial \psi^t(x) \delta x)$.
}

\subsection{Differential positivity}

{\color{black} 
A linear system on $\real^n$ is positive if 
there exists a cone $\calK\subseteq \real^n$ 
which is forward invariant for the system dynamics. Indeed,
given $\dot{x} = A x$, positivity reads $e^{At} \calK \subseteq \calK$ for all $t\geq 0$.
Differential positivity is a way to extend linear positivity 
to nonlinear dynamics, by requiring that a given cone (field) is forward invariant for the 
prolonged dynamics $\delta \Sigma$, \cite{Forni_diff_pos}. 

We endow the manifold $\calX$ with a cone field
}
\begin{equation*}
\mathcal{K}(x) \subseteq T_x \mathcal{X} \qquad \forall x \in \mathcal{X} \ .
\end{equation*}
Each cone $\mathcal{K}(x)$ is closed 
{\color{black} and solid}, and satisfies the following properties: for all $x \in \mathcal{X}$, (i) $\mathcal{K}(x)+\mathcal{K}(x) \subseteq \mathcal{K}(x)$, (ii) $\alpha \mathcal{K}(x) \subseteq \mathcal{K}(x) $ for all $\alpha \in \mathbb{R}^+$, (iii) $\mathcal{K}(x) \cap -\mathcal{K}(x)=\{0\}$
{\color{black} 
(i.e. convex and pointed). 
The forward invariance of the cone field along the prolonged dynamics $\delta \Sigma$ reads 
as follows \cite{Forni_diff_pos}.
\begin{definition}[Differential positivity]
\label{def_diff_pos}
The system $\Sigma$ is differentially positive (with respect to the cone field $\mathcal{K}$) if the flow of the prolonged system $\delta \Sigma$ leaves the cone invariant
\begin{equation*}
\partial \psi^t(x) \mathcal{K}(x) \subseteq \mathcal{K}(\psi^t(x)) \qquad \forall x\in \mathcal{X}\,, \forall t>0\,.
\end{equation*}
In addition, $\Sigma$ is (uniformly) strictly differentially positive if it is differentially positive and if there exist a constant $T>0$ and a cone field $\mathcal{R}(x) \subset \textrm{int} \mathcal{K}(x) \cup \{0\}$ such that
\begin{equation*}
\partial \psi^t(x) \mathcal{K}(x) \subseteq \mathcal{R}(\psi^t(x)) \qquad \forall x\in \mathcal{X}\,, \forall t \geq T\,. \vspace{-6mm}
\end{equation*} 
\hfill $\diamond$ \vspace{1mm}
\end{definition}
An illustration of the strict differential positivity property is 
provided in Figure \ref{fig:diff_positivity}.
}
\begin{figure}[htbp]
\centering
\includegraphics[width=0.65\columnwidth]{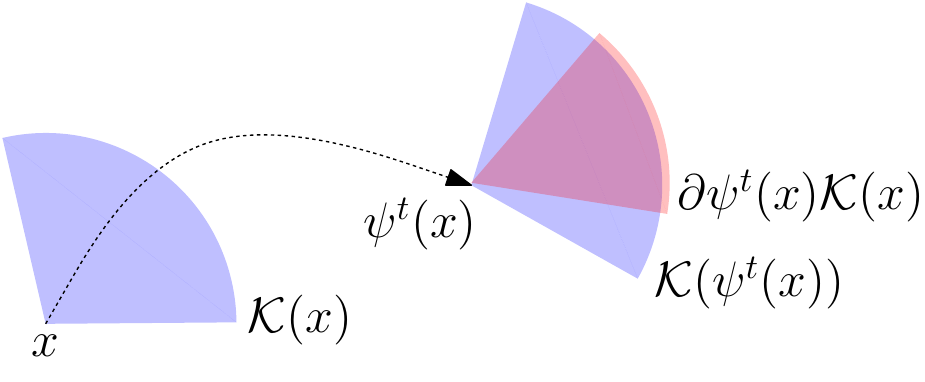}  
\caption{Differential positivity requires the forward
invariance of the cone field $\calK(x)$ along trajectories. 
Strict differential positivity requires that along trajectories 
the rays of the initial cone converge towards each other.}
\label{fig:diff_positivity}
\end{figure}

{\color{black} 
To avoid pathological cases, we assume that for every pair of points
$x_1,x_2\in \calX$, 
there exists a linear invertible mapping  
$\Gamma(x_1,x_2):T_{x_1}\calX \to T_{x_2}\calX$ 
such that 
$\Gamma(x_1,x_2)\calK(x_1) = \calK(x_2)$
and
$\Gamma(x_1,x_2)\calR(x_1) = \calR(x_2)$.
Furthermore, we consider the representation 
\begin{equation*}
\begin{array}{rcl}
\mathcal{K}(x) &=& \{\delta x \in T_x \mathcal{X} \,|\, k_i(x,\delta x) \geq 0 \quad  \forall i=1,\dots,m \} \vspace{2mm}\\
\mathcal{R}(x)&=&\{\delta x \in T_x \mathcal{X} \,|\, k_i(x,\frac{\delta x}{|\delta x|}) \geq \varepsilon \quad \forall i=1,\dots,m \} 
\end{array}
\end{equation*}
for some $m \in \mathbb{N}$ and $\varepsilon > 0$, 
where $k_i:T\calX \to \mathbb{R}$ are smooth functions 
(the reader is referred to \cite{Forni_diff_pos, Forni2015} for details).

It follows from Definition \ref{def_diff_pos} that 
$\Sigma$ is 
differentially positive if from
any initial condition $k_i(x, \delta x) \geq 0$, $i\in \{1,\dots,m\}$,
the prolonged system satisfies
$k_i(\psi^t(x),\partial \psi^t(x) \delta x) \geq 0$
for all  $i\in \{1,\dots,m\}$ and all $t \geq 0$.
In addition, strictly differential positivity requires
$k_i\left(\psi^t(x), \frac{\partial \psi^t(x) \delta x}{|\partial \psi^t(x) \delta x|}\right) \geq \varepsilon$,
for all $t \geq T$.
In compact sets, both properties can be checked by simple geometric conditions \cite{Forni2015}.

Strictly differentially positive systems enjoy a
projective contraction property \cite{Forni_diff_pos,Bushell},
which leads to the existence of the so-called 
Perron-Frobenius vector field $\mathbf{w}(x) \in \textrm{int}\calK(x)$,
the differential equivalent of the Perron-Frobenius eigenvector
of linear positive mappings. The ray 
$\{ \lambda \mathbf{w}(x) \,|\, \lambda \geq 0 \} \subset  \textrm{int}\calK(x)\cup\{0\}$ 
is an attractor for the prolonged dynamics, in the precise sense that}
\begin{equation}
\label{def_PF_vec}
\lim_{t \rightarrow \infty} \frac{\partial \psi^t(\psi^{-t}(x)) \delta x}{|\partial \psi^t(\psi^{-t}(x)) \delta x |}=\mathbf{w}(x)
\end{equation}
for all $x \in \mathcal{X}$ and $\delta x \in \mathcal{K}(\psi^{-t}(x)) \setminus \{0\}$.
{\color{black} 
It follows that, $\mathbf{w}(\psi^t(x)) =  \partial\psi ^t(x) \mathbf{w}(x) / |\partial\psi ^t(x) \mathbf{w}(x)|$.

In the next sections we will use the notion of 
Perron-Frobenius curve $\gamma^{\mathbf{w}}:I \subseteq \mathbb{R} \to \mathcal{X}$,
which is an integral curve of the Perron Frobenius vector field, i.e.  
$\frac{d}{ds}\gamma^{\mathbf{w}}(s)=\mathbf{w}(\gamma^\mathbf{w}(s))$ for all $s\in I$.
}

\subsection{Koopman operator}

The so-called Koopman operator describes the evolution of observables $g:\mathcal{X} \to \mathbb{C}$ along the trajectories of $\Sigma$.
\begin{definition}[Koopman operator]
For a given functional space $\mathcal{G}$, the (semi-)group of Koopman operators $U^t:\mathcal{G} \rightarrow \mathcal{G}$ associated with a system $\Sigma$ is defined by
\begin{equation*}
U^t g = g \circ \psi^t\,, \qquad g \in \mathcal{G} \,, t\in\mathbb{R}
\end{equation*}
where $\psi^t(\cdot)$ is the flow map of $\Sigma$.
\hfill $\diamond$
\end{definition}

Even when $\Sigma$ is a nonlinear system, the Koopman operator is linear, so that it can be studied through its spectral properties.
\begin{definition}[Koopman eigenfunction and eigenvalue]
The observable $\phi_\lambda \in \mathcal{G}$ is an eigenfunction of the Koopman operator (called Koopman eigenfunction hereafter) if there exists a value $\lambda \in \mathbb{C}$ such that
\begin{equation}
\label{phi_evol}
(\phi_{\lambda} \circ \psi^t)(x)=U^t \phi_\lambda(x)=e^{\lambda t} \phi_\lambda(x) \qquad \forall x \in \mathcal{X}\,.
\end{equation}
The value $\lambda$ is the associated eigenvalue.
\hfill $\diamond$
\end{definition}
\begin{remark}
When $\mathcal{G} \subseteq C^1(\mathcal{X})$, the semi-group of Koopman operators admits the infinitesimal generator $L:\mathcal{G} \to \mathcal{G}$ defined by $L g(x) =\partial g(x) [f(x)]$ and we have $U^t=e^{tL}$. In this case, a Koopman eigenfunction $\phi_\lambda$ satisfies
\begin{equation}
\label{PDE_Koop}
L \phi_\lambda(x) = \partial \phi_\lambda(x) [f(x)] = \lambda \phi_\lambda(x) \,. \vspace{-3mm}
\end{equation}
\hfill $\diamond$
\end{remark}

The Koopman eigenfunctions capture important geometric properties of the dynamics (see e.g. \cite{Mauroy_Mezic,MMM_isostables,Mezic}). In Section \ref{sec_gen_results}, they will be related to the differential positivity properties of the system. As a preliminary, we consider the semi-group of Koopman operators $\tilde{U}^t$ associated with the prolonged system $\delta \Sigma$, i.e.
\begin{equation*}
\tilde{U}^t \tilde{g}(x,\delta x) = \tilde{g}(\psi^t(x),\partial \psi^t(x) \delta x)
\end{equation*}
with the observables $\tilde{g}: T \mathcal{X} \to \mathbb{C}$. We have the following result.
\begin{lemma}
\label{lemma}
Suppose that $\phi_\lambda \in C^1(\mathcal{X})$ is an eigenfunction of $U^t$ associated with the system $\Sigma$. Then the Koopman operator $\tilde{U}^t$ associated with the prolonged system $\delta \Sigma$ admits the eigenfunctions (in the appropriate functional space)
\begin{eqnarray*}
\tilde{\phi}^{(1)}_\lambda(x,\delta x) & = & \phi_\lambda(x) \\
\tilde{\phi}^{(2)}_\lambda(x,\delta x) & = & \partial \phi_\lambda(x) \delta x
\end{eqnarray*}
for all $(x,\delta x) \in T\mathcal{X}$.
\end{lemma}
\begin{proof}
It is clear that $\tilde{U}^t \tilde{\phi}^{(1)}_\lambda = e^{\lambda t} \tilde{\phi}^{(1)}_\lambda$. In addition,
\begin{equation}
\label{evol_lemma}
\begin{split}
\tilde{U}^t \tilde{\phi}^{(2)}_\lambda(x,\delta x) & = \partial \phi_\lambda(\psi^t(x)) [\partial \psi^t(x)\delta x] \\
&  = \partial (\phi_\lambda \circ \psi^t)(x)\delta x \\
& = e^{\lambda t} \partial \phi_\lambda(x) \delta x \\
& = e^{\lambda t} \tilde{\phi}^{(2)}_\lambda (x,\delta x)\,,
\end{split}
\end{equation}
where we used \eqref{phi_evol}.
\end{proof}
If $\Sigma$ admits a fixed point, the prolonged system $\delta \Sigma$ is characterized by a star node (with eigenvalues of multiplicity two). In this case, it is known that the corresponding Koopman eigenvalues are also of multiplicity two (see Remark 2 in \cite{MMM_isostables}).

{\color{black}
\section{From Koopman eigenfunctions to cone fields}
\label{sec_gen_results}
}

In this section, we present general results that connect differential positivity to the spectral properties of the Koopman operator. 
We show that a system is differentially positive if 
there exist specific independent Koopman eigenfunctions.
\begin{proposition}
\label{prop_gen}
Suppose that the $n$-dimensional system $\Sigma$ admits a set of Koopman eigenfunctions $\phi_{\lambda_j}\in C^1(\mathcal{X})$, $j=1,\dots,n${, $\Re\{\lambda_1\} \geq \Re\{\lambda_{j}\}$,} such that the linear map 
$\partial \Phi(x):T_x \mathcal{X} \to \mathbb{C}^n$, 
$$\partial \Phi(x) \delta x = (\partial \phi_{\lambda_1}(x) \delta x,\dots,\partial \phi_{\lambda_n}(x) \delta x)$$ is injective for all $x \in \mathcal{X}$. Then $\Sigma$ is differentially positive if one of the following is satisfied:
{\color{black}
\begin{enumerate}
\item $\lambda_1 \in \mathbb{R}$. The cone field reads $\mathcal{K}(x)$
\begin{equation}
\label{cone_phi}
\hspace{-5mm}\left\{\delta x \!\in\! \mathcal{T}_x \mathcal{X} |\, \partial \phi_{\lambda_1}(x)\delta x - |\partial \phi_{\lambda_j}(x)\delta x | \geq 0,\, \forall j \geq 2 \right\} ;
\end{equation}
\item $\lambda_1 \in i\mathbb{R}$ with $\angle \phi_{\lambda_1} \in C^1(\mathcal{X})$ and $|\phi_{\lambda_1}|$ is constant on $\mathcal{X}$.
The cone field reads $\mathcal{K}(x)$
\begin{equation}
\label{cone_phi2}
\hspace{-5mm}\left\{\delta x \in \mathcal{T}_x \mathcal{X}
\  |\, \partial \angle \phi_{\lambda_1}(x)\delta x - |\partial \phi_{\lambda_j}(x)\delta x | \geq 0 \,, \ \forall j \geq 2 \right \} .
\end{equation}
\end{enumerate}
}
\noindent
The system is strictly differentially positive if $\Re\{\!\lambda_1\!\} \!\!>\! \Re\{\!\lambda_2\!\}$.
\end{proposition}

{
Every cone field $\mathcal{K}$ 
is the local representation of a (global) conal order
$\prec\,\,\subseteq\!\calX\!\times\!\calX$.
The order $\prec$ is derived from $\mathcal{K}$  by integration:
$x_1 \prec x_2$ if and only if there exists a curve $\gamma:[s_1,s_2] \subseteq \mathbb{R} \to \mathcal{X}$ with $\gamma(s_1)=x_1$ and $\gamma(s_2)=x_2$ such that 
$\frac{d}{ds}\gamma(s) \in \mathcal{K}(\gamma(s))$
for all $s \in [s_1,s_2]$. 
It is noticeable that the conal order 
given by the cone field of Proposition \ref{prop_gen}
has the following equivalent characterization:
$x_1 \!\prec\! x_2$ if and only if
$\phi_{\lambda_1}(x_2)-\phi_{\lambda_1}(x_1) + |\phi_{\lambda_j}(x_2)-\phi_{\lambda_j}(x_1)| \!>\! 0$, for all  $j\geq 2$ (replace $\phi_{\lambda_1}$ by $\angle \phi_{\lambda_1}$ in case 2)).
Indeed, the contracting cone field defined locally
by the Koopman eigenfunctions of the prolonged system $\delta \Sigma$
induces a conal order 
which is captured directly by the Koopman eigenfunctions of $\Sigma$.

\begin{remark}
When $\Re\{\lambda_j\}<0$ for all $j$ in Proposition \ref{prop_gen},
the Koopman eigenfunctions of $\delta \Sigma$ can be used to
construct a differential Finsler-Lyapunov function
which decays along the trajectories of the system.
By integration, a differential Finsler-Lyapunov function 
induces a distance on the system state manifold, which
also decays along any pair of trajectories 
of $\Sigma$, establishing contraction \cite{Forni}.
As above, it is noticeable that such distance has an equivalent
characterization based on the Koopman eigenfunctions of $\Sigma$.
Indeed, the Koopman eigenfunctions of $\Sigma$ capture the contractive
behavior of the system \cite{Mauroy_Mezic_stability,MMM_isostables}
\end{remark}

The next proposition illustrates}
the relationship between Koopman eigenfunctions 
and the Perron-Frobenius vector field.
\begin{proposition}
\label{prop_PF_vec}
Suppose that there exists a set of eigenfunctions $\phi_{\lambda_j}$ that satisfies the conditions of Proposition \ref{prop_gen} with $\Re\{\lambda_1\}> \Re\{\lambda_j\}$ for all $j$ (i.e. $\Sigma$ is strictly differentially positive). Then the Perron-Frobenius vector field is the unique vector field $w:\mathcal{X} \to T_x \mathcal{X}$, $|\mathbf{w}(\cdot)|=1$, that satisfies
\begin{equation}
\label{cond_PF_vec}
\partial \phi_{\lambda_j}(x) [\mathbf{w}(x)]=0 \qquad j\geq 2\,. \vspace{-3mm}
\end{equation}
\end{proposition}

Proposition \ref{prop_PF_vec} implies that the Perron-Frobenius vector field is related to zero level sets of Koopman eigenfunctions associated with the prolonged system $\delta \Sigma$. At global level, the integral curves of the Perron-Frobenius vector field correspond to the intersection of the level sets of Koopman eigenfunctions associated with $\Sigma$. 
{Precisely, any set
\begin{equation}
\label{PF_curve}
\bigcap_{j=2}^m \{x \in \mathcal{X} |\phi_{\lambda_j}(x)=C_j  \}\ , \ \ 
(C_2,\dots,C_m) \in \mathbb{C}^{n-1}
\end{equation}
is the image of some Perron-Frobenius curve.
}

\section{Converse results for differential positivity}
\label{sec_converse}

{
Restricting the analysis to systems $\Sigma$ on 
vector spaces $\calX := \mathbb{R}^n$,
we show that the presence of a stable hyperbolic fixed point or 
of a stable hyperbolic limit cycle is a sufficient 
condition for $\Sigma$ to be strictly differentially positive 
in their basin of attraction.
}


\begin{proposition}[{Hyperbolic stable fixed point}]
\label{prop_fx_pt}
Consider a system $\dot{x}=f(x)$, with $x\in \mathbb{R}^n$ and $f \in C^2$, which admits a fixed point $x^*$ with a basin of attraction $\mathcal{B}(x^*) \subseteq \mathbb{R}^n$. 
{ For $j=1,\dots,n$}, assume that the eigenvalues $\lambda_j$
of the Jacobian matrix $J=\frac{\partial f}{\partial x}(x^*)$ satisfy $0>\Re\{\lambda_j\} \geq \Re\{\lambda_{j+1}\}$ and that the eigenvectors are independent. The system is differentially positive in $\mathcal{B}(x^*)$ if and only if $\lambda_1 \in \mathbb{R}$. Moreover, it is strictly differentially positive if $\lambda_1 > \Re\{\lambda_2\}$.
\end{proposition}

The result of Proposition \ref{prop_fx_pt} also holds if the fixed point is unstable with $\Re\{\lambda_j\}>0$ for all $j$
{ but it does not hold with a saddle node.
This is not surprising since \cite[Corollary 3]{Forni_diff_pos}
shows examples of hyperbolic saddle nodes that are incompatible
with differential positivity.}


\begin{proposition}[{Hyperbolic stable} limit cycle]
\label{prop_lim_cyc}
If a system $\dot{x}=f(x)$, with $x\in \mathbb{R}^n$ and $f \in C^2$, admits a stable hyperbolic limit cycle $\Gamma$ (with independent eigenvectors of the monodromy matrix), then it is strictly differentially positive in the basin of attraction $\mathcal{B}(\Gamma) \subseteq \mathbb{R}^n$ of $\Gamma$.
\end{proposition}

\begin{remark}
{Proposition \ref{prop_PF_vec} and \eqref{PF_curve}}
show that the Perron-Frobenius curves are the intersections of the level sets of $n-1$ Koopman eigenfunctions $\phi_{\lambda_j}$, $j=2,\dots,n$. 
{ Therefore, the Perron-Frobenius curves 
can be interpreted as dual quantities to }
the $(n-1)$-dimensional level sets of the Koopman eigenfunction $\phi_{\lambda_1}$, 
{ the so-called \emph{isostables} for 
fixed points \cite{MMM_isostables} and the \emph{isochrons} for
limit cycles \cite{Mauroy_Mezic}}. \hfill 
$\diamond$
\end{remark}

\section{Numerical computation}
\label{sec_numerical}
{ 
\subsection{Preliminaries on Laplace averages}

In this section we exploit the theoretical results of the paper 
to derive contracting cone fields for fixed points
and limit cycles based on the
Koopman eigenfunctions of the prolonged system.
}


When the trajectories of the system are available, an efficient method for computing Koopman eigenfunctions is based on Laplace averages (see e.g. \cite{Mohr})
\begin{equation}
\label{Lap_av}
g^{av}_\lambda(x)=\lim_{T \rightarrow \infty} \frac{1}{T} \int_0^T (g \circ \psi^t)(x) e^{-\lambda t} \, dt
\end{equation}
for an observable $g:\mathcal{X} \to \mathbb{C}$. 
{ The average \eqref{Lap_av} is well-defined 
(i.e. finite) when 
$\lambda$ is a Koopman eigenvalue and 
$g$ is a well-chosen observable (for instance, $g$ must be zero 
on the attractor when $\Re\{\lambda\}<0$).
From \eqref{Lap_av},}
it is easy to see that $U^t g^{av}_\lambda=e^{\lambda t} g^{av}_\lambda$.
Provided that $g^{av}_\lambda \neq 0$, we can define $\phi_\lambda \triangleq g^{av}_\lambda$.

We can also obtain the Koopman eigenfunction $\tilde{\phi}^{(2)}_\lambda(x,\delta x)=\partial \phi_\lambda(x) \delta x$ associated with the prolonged system $\delta \Sigma$ by computing the Laplace averages along the trajectories of $\delta \Sigma$. This is summarized in the following lemma.
\begin{lemma}
\label{lem2}
Suppose that the Laplace average $g^{av}_\lambda(x)$ is finite and nonzero, so that $\phi_\lambda \triangleq g^{av}_\lambda$. Then, we have
\begin{equation*}
\tilde{\phi}^{(2)}_\lambda(x,\delta x) \triangleq \tilde{g}^{av}_\lambda(x,\delta x)\,,
\end{equation*}
with
\begin{equation}
\label{Lap_av_prolonged}
\tilde{g}^{av}_\lambda(x,\delta x)=\lim_{T \rightarrow \infty} \frac{1}{T} \int_0^T \tilde{g}(\psi^t(x),\partial \psi^t(x)\delta x) e^{-\lambda t} \, dt
\end{equation}
and with the observable $\tilde{g}(x,\delta x)=\partial g(x) \delta x$.
\end{lemma}
\begin{proof}
The result is obtained by differentiating \eqref{Lap_av}.
\end{proof}

{ 
\subsection{Fixed points}
}
When the system admits a stable hyperbolic fixed point $x^*$, a contracting cone field is given by \eqref{cone_phi},
{
where the Koopman eigenfunctions are given by the Laplace averages \eqref{Lap_av_prolonged}.
}
{Figure \ref{fig:conefield_fx_pt} shows the contracting cone field for the
dynamics
\begin{equation}
\label{syst1}
\begin{array}{rcl}
\dot{x}_1 & = & -\sin(x_1) + \cos(x_2) -1  \vspace{1mm} \\
\dot{x}_2 & = & -\cos(x_1) - 1.5 \sin(x_2) +1
\end{array}
\end{equation}
which has a stable fixed point at the origin.} 

The Laplace averages $\tilde{g}^{av}_{\lambda_1}$ and $\tilde{g}^{av}_{\lambda_2}$ (where $\lambda_1=-1$ and $\lambda_2=-1.5$ are the eigenvalues of the Jacobian matrix $J$ at the origin) are computed with the observables $\tilde{g}(x,\delta x)=v_1^T \delta x$ and $\tilde{g}(x,\delta x)=v_2^T \delta x$, respectively, where $v_1$ and $v_2$ are the left eigenvectors of $J$. This choice ensures that each average $\tilde{g}^{av}_{\lambda_j}$ is finite (provided also that $2\lambda_1<\lambda_2<\lambda_1$, which is the case here) and nonzero. Note also that we have $\tilde{g}=\partial g$ with $g=v_j^T x$, as required by Lemma \ref{lem2}. The Perron-Frobenius vector field is given by \eqref{cond_PF_vec}. An illustration is in Figure \ref{fig:conefield_fx_pt}.

\begin{figure}[htbp]
\centering
\includegraphics[width=0.8\columnwidth]{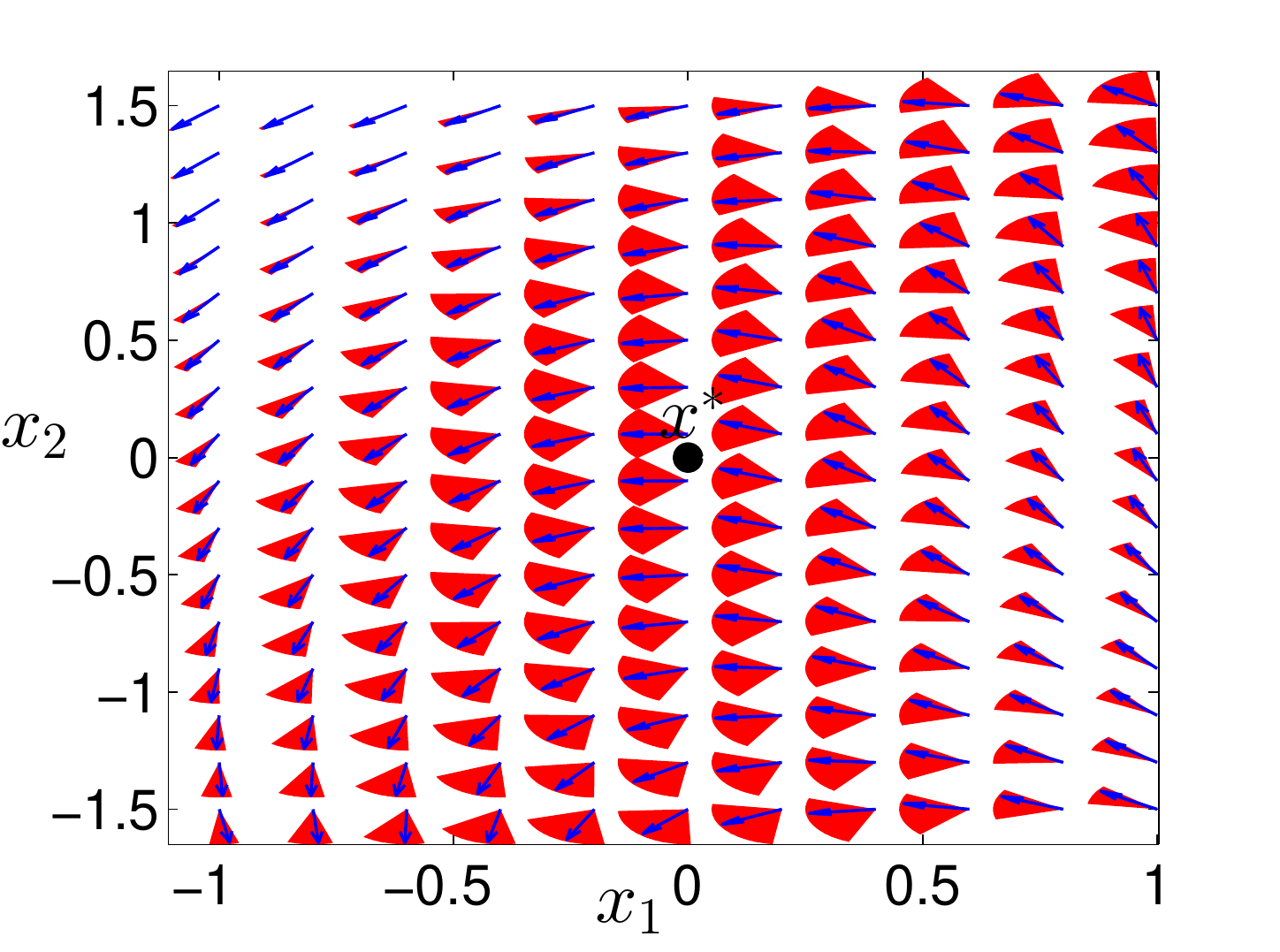} \vspace{-3mm}
\caption{Cone field (red) and Perron Frobenius vector field (blue) for \eqref{syst1}.}
\label{fig:conefield_fx_pt}
\end{figure}

{
\subsection{Limit cycles}
} 
When the system admits a stable hyperbolic limit cycle, a contracting cone field is defined by \eqref{cone_phi2} and can be expressed in terms of Koopman eigenfunctions associated with the system $\Sigma$ and the prolonged system $\delta \Sigma$. The Laplace averages \eqref{Lap_av} and \eqref{Lap_av_prolonged} can be used to compute the cone field.

{Figure \ref{fig:conefield_vdp} shows
the contracting cone field for the Van der Pol dynamics
\begin{equation}
\label{syst2}
\begin{array}{rcl}
\dot{x}_1 & = &  x_2  \\
\dot{x}_2 & = & (1-x_1^2) x_2 - x_1
\end{array}
\end{equation}
which has a stable limit cycle $\Gamma$ (of period $T>0$)}. 

The Laplace averages $g^{av}_{\lambda_1}$ and $\tilde{g}^{av}_{\lambda_1}$ (with $\lambda_1=i 2\pi/T$) are computed with the observables $g(x)=[\,1 \ 0\,]^T x$ and $\tilde{g}(x,\delta x)=\partial g(x)\delta x=[\,1 \ 0\,]^T \delta x$, respectively. According to \eqref{partial_angle}, we have $\partial \angle \phi_{\lambda_1}=\tilde{g}^{av}_{\lambda_1}/(i g^{av}_{\lambda_1})$.
The average $\tilde{g}^{av}_{\lambda_2}$ (where $\lambda_2$ is the nonzero Floquet exponent of the limit cycle) is computed with the observable $\tilde{g}(x,\delta x)=\xi(\rho(x))^T \delta x$, where $\xi:\Gamma \to T\mathcal{X}$ is a unit vector field such that $\xi(x)$ is perpendicular to the tangent direction to the limit cycle at $x$, and where $\rho:\mathcal{X} \to \Gamma$ is a radial projection on $\Gamma$ (i.e. $\rho(x)$ is the intersection between $\Gamma$ and the line passing through $x$ and the origin). Note that, as required by Lemma \ref{lem2}, $\tilde{g}$ is the differential of an observable measuring a distance to the limit cycle. The Perron-Frobenius vector field is given by \eqref{cond_PF_vec}.
{An illustration is in Figure \ref{fig:conefield_vdp}.}

\begin{figure}[htbp]
\centering
\includegraphics[width=0.8\columnwidth]{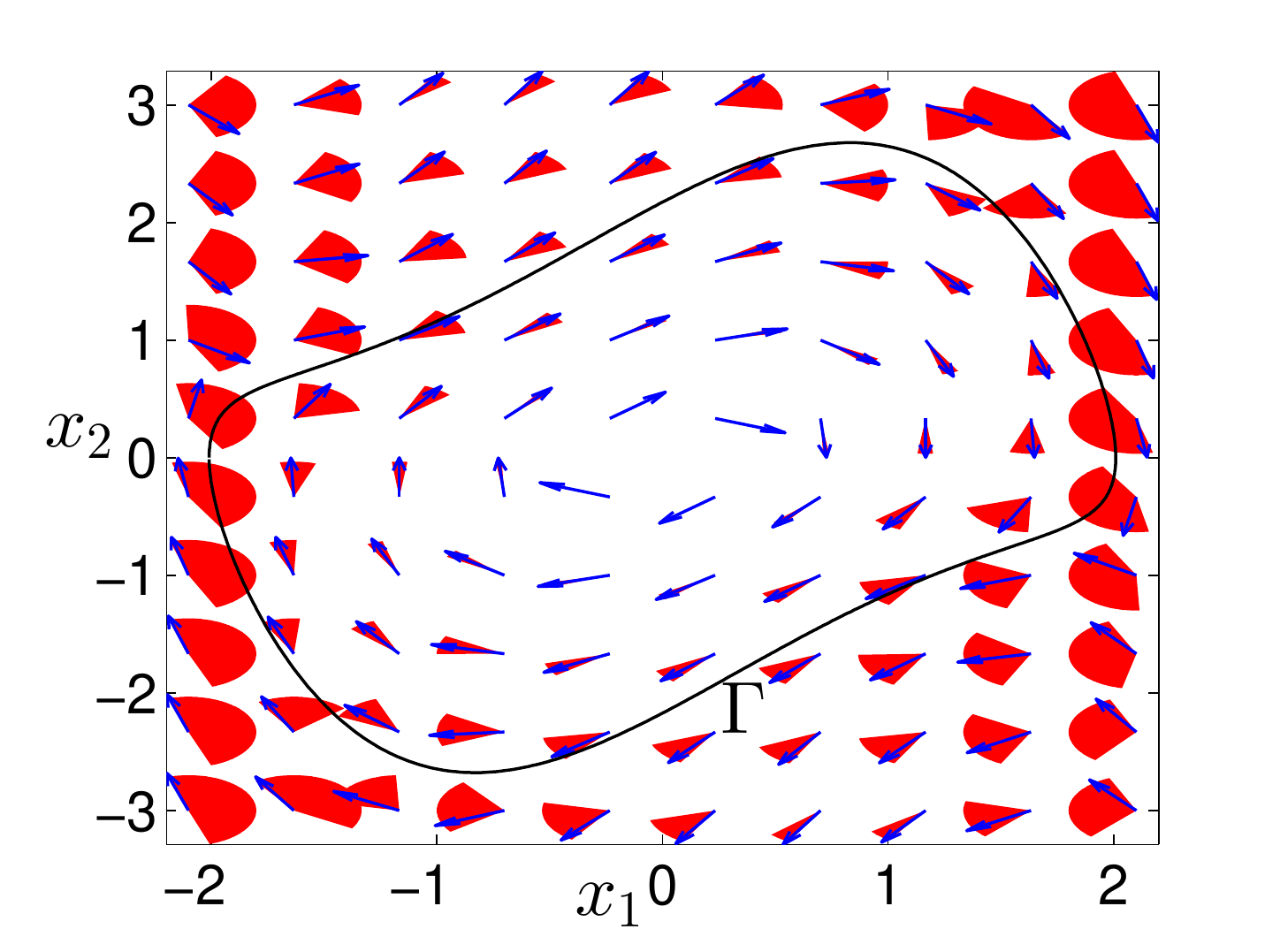} \vspace{-3mm}
\caption{Cone field (red) and Perron Frobenius vector field (blue) for \eqref{syst2}.}
\label{fig:conefield_vdp}
\end{figure}

{
\begin{remark}
Besides the Laplace averages, Koopman eigenfunctions can be computed
by other methods, which also provide novel ways to derive
cone fields.  For example, the Koopman operator can be expanded on a (finite) polynomial basis \cite{Mauroy_Mezic_stability}, yielding a polynomial approximation for
eigenfunctions and cone fields. In a similar way, 
the extended dynamic mode decomposition \cite{Rowley_EDMD} could also be 
employed.
\end{remark}
}

\section{Conclusion}
\label{conclu}
{\color{black}
Reducing the analysis of a nonlinear finite-dimensional system 
to the analysis of an infinite-dimensional linear system,
Koopman operator theory 
is a powerful tool of nonlinear control.
In this paper we bridged Koopman operator theory and
differential positivity, opening the way to the use
of spectral methods for differential positivity.
We illustrated a tight relation between Koopman eigenfunctions
and cone fields, leading to converse results for differential positivity
and to numerical tools for the construction of cone fields.

The bridge between Koopman theory and
differential positivity opens new interesting directions
of research. For example, we believe that the approach 
pursued in the paper will lead to  
converse results for the larger class of normally
hyperbolic one-dimensional attractors. 
This will be the object of future research.
}

\section{Acknowledgments}

A. Mauroy holds a BELSPO Return Grant and F. Forni holds a FNRS fellowship. This paper presents research results of the Belgian Network DYSCO, funded by the Interuniversity Attraction Poles Programme initiated by the Belgian Science Policy Office.

\appendix

\begin{proofof}\emph{Proposition \ref{prop_gen}}.

\emph{Case $\lambda_1 \in \mathbb{R}$.}
We first show that \eqref{cone_phi}
is a well-defined cone field. (i) $\mathcal{K}(x)+\mathcal{K}(x) \subseteq \mathcal{K}(x)$ since
$\partial \phi_{\lambda_1}(x)[\delta x + \delta x'] - |\partial \phi_{\lambda_j}(x)[\delta x +\delta x'] | \geq \partial \phi_{\lambda_1}(x)\delta x + \partial \phi_{\lambda_1}(x)\delta x' - |\partial \phi_{\lambda_j}(x)\delta x | - |\partial \phi_{\lambda_j}(x)\delta x' | \geq 0$ if $\delta x,\delta x'\in \mathcal{K}(x)$. (ii) $\alpha \mathcal{K}(x) \subseteq \mathcal{K}(x)$, $\alpha>0$ since $\partial \phi_{\lambda_1}(x)[\alpha \delta x] - |\partial \phi_{\lambda_j}(x)[\alpha \delta x] | \geq 0$ if $\delta x\in \mathcal{K}(x)$. (iii) $\mathcal{K}(x) \cap -\mathcal{K}(x) = \{0\}$ since $\partial \phi_{\lambda_1}(x)\delta x - |\partial \phi_{\lambda_j}(x)\delta x | \geq 0$ and $-\partial \phi_{\lambda_1}(x)\delta x - |\partial \phi_{\lambda_j}(x)\delta x | \geq 0$ imply $\partial \phi_{\lambda_j}(x)\delta x=0$ $\forall j$, so that $\delta x=0$ since the linear map $\partial \Phi(x)$ is injective.

Using the Koopman eigenfunctions of the prolonged system $\delta \Sigma$ and Lemma \ref{lemma}, we have
\begin{equation}
\begin{split}
\label{inequa3}
& \partial \phi_{\lambda_1}(\psi^t(x),\partial \psi^t(x)\delta x) - |\partial \phi_{\lambda_j}(\psi^t(x),\partial \psi^t(x)\delta x) | \\
& = \tilde{U}^t \tilde{\phi}^{(2)}_{\lambda_1} (x,\delta x) - |\tilde{U}^t \tilde{\phi}^{(2)}_{\lambda_j} (x,\delta x) | \\
& = e^{\lambda_1 t} \left(\tilde{\phi}^{(2)}_{\lambda_1}(x,\delta x) - e^{(\Re\{\lambda_j\}-\lambda_1)t} | \tilde{\phi}^{(2)}_{\lambda_j}(x,\delta x) | \right) \\
& \geq e^{\lambda_1 t} \left(\partial \phi_{\lambda_1}(x) \delta x - | \partial \phi_{\lambda_j}(x) \delta x | \right) 
\\
& \geq 0
\end{split}
\end{equation}
for all $x\in \mathcal{X}$ and $\delta x \in \mathcal{K}(x)$ and for all $t>0$. It follows that $\partial \psi^t(x) \mathcal{K}(x) \subseteq \mathcal{K}(\psi^t(x))$. If $\Re\{\lambda_1\} > \Re\{\lambda_2\}$, \eqref{inequa3} is a strict inequality and $\Sigma$ is strictly differentially positive (uniformly with an arbitrary $T>0$).

\emph{Case $\lambda_1 \in i\mathbb{R}$.} We consider the cone field \eqref{cone_phi2}.
It follows on similar lines that (i) $\mathcal{K}(x)+\mathcal{K}(x) \subseteq \mathcal{K}(x)$ and (ii) $\alpha \mathcal{K}(x) \subseteq \mathcal{K}(x)$, $\alpha>0$. In addition, $\partial \angle \phi_{\lambda_1}(x)\delta x - |\partial \phi_{\lambda_j}(x)\delta x | \geq 0$ and $-\partial \angle \phi_{\lambda_1}(x)\delta x - |\partial \phi_{\lambda_j}(x)\delta x | \geq 0$ imply $\partial \angle \phi_{\lambda_1}(x)\delta x=\partial \phi_{\lambda_j}(x)\delta x=0$ $\forall j \geq 2$. Since $|\phi_{\lambda_1}|$ is constant by assumption, we have
\begin{equation}
\label{property_angle}
\partial \phi_{\lambda_1}  = \partial \left(|\phi_{\lambda_1}| e^{i \angle \phi_{\lambda_1}} \right) = i \phi_{\lambda_1} \partial \angle \phi_{\lambda_1}
\end{equation}
and it follows that $\partial \phi_{\lambda_1}(x)\delta x=0$. Then the injectivity of $\partial \Phi(x)$ implies $\delta x=0$ , so that $\mathcal{K}(x) \cap -\mathcal{K}(x) = \{0\}$.

Since $|\phi_{\lambda_1}|$ must be nonzero, it follows from \eqref{property_angle} that
\begin{equation}
\label{partial_angle}
\partial \angle \phi_{\lambda_1}(x) \delta x = \frac{\partial \phi_{\lambda_1}(x) \delta x}{i \phi_{\lambda_1}(x)} = \frac{\tilde{\phi}^{(2)}_{\lambda_1}(x,\delta x)}{i \phi_{\lambda_1}(x)}\,.
\end{equation}
Using Lemma \ref{lemma} and $\Re\{\lambda_j\}\leq \Re\{\lambda_1\}=0$, we have
\begin{equation*}
\begin{split}
& \partial \angle \phi_{\lambda_1}(\psi^t(x))[\partial \psi^t(x)\delta x] - |\partial \phi_{\lambda_j}(\psi^t(x))[\partial \psi^t(x)\delta x] | \\
& = \frac{\tilde{U}^t \tilde{\phi}^{(2)}_{\lambda_1}(x,\delta x)}{i U^t \phi_{\lambda_1}(x)} - |\tilde{U}^t \tilde{\phi}^{(2)}_{\lambda_j} (x,\delta x) | \\
& = \frac{\tilde{\phi}^{(2)}_{\lambda_1}(x,\delta x)}{i \phi_{\lambda_1}(x)} - e^{\Re\{\lambda_j\} t} | \tilde{\phi}^{(2)}_{\lambda_j}(x,\delta x) | \\
& \geq \partial \angle \phi_{\lambda_1}(x)\delta x - | \partial \phi_{\lambda_j}(x) \delta x | \\
& \geq 0
\end{split}
\end{equation*}
for all $x\in \mathcal{X}$ and $\delta x \in \mathcal{K}(x)$ and for all $t>0$. It follows that $\partial \psi^t(x) \mathcal{K}(x) \subseteq \mathcal{K}(\psi^t(x))$. The proof for strictly differential positivity follows on similar lines.
\end{proofof} \vspace{3mm}

\begin{proofof}\emph{Proposition \ref{prop_PF_vec}.}

Consider $\alpha \in \mathbb{R}$ such that $|\partial \phi_{\lambda_1}(x)[\alpha \mathbf{w}(x)]|=1$. If $\lambda_1 \in \mathbb{R}$, this implies that $\partial \phi_{\lambda_1}(x)[\alpha \mathbf{w}(x)]$ has a unique value up to a sign. If $\lambda_1 \in i\mathbb{R}$,  $|\phi_{\lambda_1}|$ is constant and \eqref{property_angle} implies that $\partial \phi_{\lambda_1}(x)[\alpha \mathbf{w}(x)]$ has a unique value up to a sign. Since $\partial \Phi$ is injective, it follows that $\mathbf{w}(x)$ is unique (up to a sign).

Next, we show that the Perron-Frobenius vector field \eqref{def_PF_vec} satisfies \eqref{cond_PF_vec}. We have
\begin{equation}
\label{dev_PF}
\begin{split}
|\partial & \phi_{\lambda_j}(x)[\partial \psi^t(\psi^{-t}(x)) \delta x] | \\
& = |\partial \phi_{\lambda_j}(\psi^t(\psi^{-t}(x)))[\partial \psi^t(\psi^{-t}(x)) \delta x] | \\
& = |\tilde{U}^t \tilde{\phi}^{(2)}_{\lambda_j} (\psi^{-t}(x), \delta x)|\\
& = | e^{\lambda_j t} \tilde{\phi}^{(2)}_{\lambda_j} (\psi^{-t}(x), \delta x) | \\
&= e^{\Re\{\lambda_j\} t} |\partial \phi_{\lambda_j}(\psi^{-t}(x)) \delta x |
\end{split}
\end{equation}
where we used $x=\psi^t(\psi^{-t}(x))$ and Lemma \ref{lemma}. If $\lambda_1 \in \mathbb{R}$ and $\delta x \in \mathcal{K}(\psi^{-t}(x))$, it follows from \eqref{dev_PF} that
\begin{equation*}
\begin{split}
|\partial & \phi_{\lambda_j}(x)[\partial \psi^t(\psi^{-t}(x)) \delta x] | \\
& \leq e^{\Re\{\lambda_j\} t} \partial \phi_{\lambda_1}(\psi^{-t}(x)) \delta x \\
& = e^{\Re\{\lambda_j\} t} \tilde{\phi}_{\lambda_1}^{(2)}(\psi^{-t}(x),\delta x) \\
& = e^{(\Re\{\lambda_j\}-\lambda_1) t} \tilde{U}^t \tilde{\phi}_{\lambda_1}^{(2)}(\psi^{-t}(x),\delta x) \\
& = e^{(\Re\{\lambda_j\}-\lambda_1) t} \partial \phi_{\lambda_1}(x) [\partial \psi^t(\psi^{-t}(x)) \delta x] \\
& \leq e^{(\Re\{\lambda_j\}-\lambda_1) t} \|\partial \phi_{\lambda_1}(x)\| |\partial \psi^t(\psi^{-t}(x)) \delta x|
\end{split}
\end{equation*}
with $j \geq 2$ and with $\|\partial \phi_{\lambda_1}(x)\| = \max_{|\delta x|=1} |\partial \phi_{\lambda_1}(x) \delta x|$. In the case $\lambda_1 \in i\mathbb{R}$, it follows on similar lines that
\begin{equation*}
\begin{split}
|\partial \phi_{\lambda_j}(x) & [\partial \psi^t(\psi^{-t}(x)) \delta x] | \\
&  \leq e^{\Re\{\lambda_j\} t} \|\partial \angle \phi_{\lambda_1}(x)\| |\partial \psi^t(\psi^{-t}(x)) \delta x|\,.
\end{split}
\end{equation*}
Assuming without loss of generality that $|\phi_{\lambda_1}(x)|=1$ when $\lambda_1 \in i\mathbb{R}$, \eqref{property_angle} implies that $\|\partial \angle \phi_{\lambda_1}\|=\|\partial \phi_{\lambda_1}\|$. Then, in both cases $\lambda_1 \in \mathbb{R}$ and $\lambda_1 \in i\mathbb{R}$, we have
\begin{equation*}
\frac{|\partial \phi_{\lambda_j}(x)[\partial \psi^t(\psi^{-t}(x)) \delta x] |}{|\partial \psi^t(\psi^{-t}(x)) \delta x|} \leq e^{(\Re\{\lambda_j\}-\Re\{\lambda_1\}) t} \|\partial \phi_{\lambda_1}(x)\|\,.
\end{equation*}
Finally, taking the limit $t\rightarrow \infty$ and using the definition \eqref{def_PF_vec}, we obtain $|\partial \phi_{\lambda_j}(x) [\mathbf{w}(x)]|=0$ since $\Re\{\lambda_j\}<\Re\{\lambda_1\}$. This concludes the proof.
\end{proofof} \vspace{3mm}

\begin{proofof}\emph{Proposition \ref{prop_fx_pt}}

\emph{Sufficiency.} It is known that the system admits $n$ independent Koopman eigenfunctions $\phi_{\lambda_j} \in C^1(\mathcal{B}(x^*))$ associated with the eigenvalues $\lambda_j$ of $J$. This result follows from Theorem 2.3 in \cite{Lan}, which shows the existence of a $C^1$ diffeomorphism $h:\mathcal{B}(x^*) \rightarrow \mathbb{R}^n$ such that $y=h(x)$ and
$\dot{y}=J\,y$. Letting $\phi_{\lambda_j}(x)=v_j^T h(x)$ where $v_j$ is the left eigenvector of $J$ associated with $\lambda_j$, we verify that
\begin{equation*}
\partial \phi_{\lambda_j}(x)[f(x)] = v_j^T \partial h(x)[f(x)] = v_j^T J y = \lambda_j v_j^T y = \lambda_j \phi_{\lambda_j}
\end{equation*}
and \eqref{PDE_Koop} implies that $\phi_{\lambda_j}$ is a Koopman eigenfunction. In addition, since $h$ is a diffeomorphism, $\partial h$ is injective and the linear map $\partial \Phi(x)=(v_1^T \partial h(x),\dots,v_n^T \partial h(x))$ is injective since the eigenvectors $v_j$ are independent. Then the result follows from Proposition \ref{prop_gen} with $\mathcal{X}=\mathcal{B}(x^*)$ and $\lambda_1 \in \mathbb{R}$.

\emph{Necessity.} If the system is (strictly) differentially positive, the differential dynamics $\dot{\delta x}=J\, \delta x$ at the fixed point is (strictly) differentially positive. Then the invariant cone $\mathcal{K}(x^*)$ must contain the dominant direction, which corresponds to the right eigenvector $v$ of $J$ associated with $\lambda_1$. If $\lambda_1 \notin \mathbb{R}$, it is clear that the invariance of the cone field implies that $\mathcal{K}(x^*)$ contains the entire two-dimensional plane spanned by $\Re\{v\}$ and $\Im\{v\}$. Thus we have $\mathcal{K}(x^*) \cap -\mathcal{K}(x^*) \neq \{0\}$ and $\mathcal{K}(x^*)$ is not a pointed cone. This is a contradiction, so that $\lambda_1 \in \mathbb{R}$.
\end{proofof} \vspace{3mm}

\begin{proofof}\emph{Proposition \ref{prop_lim_cyc}.}

It is known that the system admits $n$ independent Koopman eigenfunctions $\phi_{\lambda_j} \in C^1(\mathcal{B}(x^*))$. This result follows from Theorem 2.6 in \cite{Lan} (and from Floquet theory), which shows the existence of a $C^1$ diffeomorphism $h=(h_y,h_\theta):\mathcal{B}(\Gamma) \rightarrow \mathbb{R}^{n-1} \times \mathbb{S}^1$ such that $(y,\theta)=(h_y(x),h_\theta(x))$ and $\dot{y} = B\,y$, $\dot{\theta} = \omega$, where $\omega=2\pi/T$ (with $T$ the period of the limit cycle). It is clear that $\phi_{\lambda_1}(x)=e^{i h_\theta(x)}$ is a Koopman eigenfunction associated with the eigenvalue $\lambda_1=i \omega$. In addition, the matrix $B$ has $n-1$ eigenvalues $\lambda_j$, $j=2,\dots,n$ (i.e. the nonzero Floquet exponents of the limit cycle) associated with the left eigenvectors $v_j$. Since the limit cycle is hyperbolic and stable, we have $\Re\{\lambda_j\}<0$ for all $j \geq 2$. As in the proof of Proposition \ref{prop_fx_pt}, it follows that $\phi_{\lambda_j}(x)=v_j^T h_y(x)$ are Koopman eigenfunctions associated with the eigenvalues $\lambda_j$. In addition, since $h$ is a diffeomorphism, $\partial h$ is injective and the linear map $\partial \Phi(x)=(i e^{i h_\theta(x)} \partial h_\theta(x),v_2^T \partial h_y(x),\dots,v_n^T \partial h_y(x))$ is injective since the eigenvectors $v_j$ are independent.  Finally, $\angle \phi_{\lambda_1}=h_\theta \in C^1$ and $|\phi_{\lambda_1}|=1$. Then the result follows from Proposition \ref{prop_gen} with $\mathcal{X}=\mathcal{B}(x^*)$ and $\lambda_1 \in i\mathbb{R}$.
\end{proofof}



\bibliographystyle{plain}

\end{document}